\documentclass[10pt,reqno]{amsart}
\usepackage{latexsym,amsxtra,amscd,ifthen}
\usepackage{amsfonts}
\usepackage{verbatim}
\usepackage{amsmath}
\usepackage{amsthm}
\usepackage{url}
\usepackage{multirow}
\usepackage{array}
\usepackage{setspace}
\usepackage[numbers, sort & compress]{natbib}
\usepackage{amssymb}

\usepackage{etex}
\usepackage{subfigure}
\usepackage{xmpmulti}
\usepackage{colortbl,dcolumn}
\usepackage[all,cmtip,line]{xy}

\usepackage{tikz}

\usetikzlibrary{calc}
\usetikzlibrary{through}

\newtheorem{proposition}{Proposition}[section]
\newtheorem{lemma}[proposition]{Lemma}

\newtheorem{remark}[proposition]{Remark}

\theoremstyle{definition}
\newtheorem{definition}[proposition]{Definition}

\newcommand{\selabel}[1]{\label{se:#1}}

\def\<{\leqslant}
\def\>{\geqslant}
\def\a{\alpha}
\def\b{\beta}
\def\d{\delta}
\def\g{\gamma}

\def\o{\omega}

\def\e{\varepsilon}

\def\L{\Lambda}
\def\s{\sigma}

\def\ot{\otimes}
\def\ra{\rightarrow}

\date{}

\begin{document}
\title{Notes on  semisimple tensor categories of rank two}
\author{Hua Sun}
\address{College of Mathematical Science, Yangzhou University,
Yangzhou 225002, China}
\email{huasun@yzu.edu.cn}
\author{Hui-Xiang Chen}
\address{College of Mathematical Science, Yangzhou University,
Yangzhou 225002, China}
\email{hxchen@yzu.edu.cn}
\author{Yinhuo Zhang}
\address{Department of Mathematics $\&$  Statistics, University of Hasselt, Universitaire Campus, 3590 Diepeenbeek,Belgium}
\email{yinhuo.zhang@uhasselt.be}
\subjclass[2010]{18D10, 16T05}
\keywords{Green ring, Auslander algebra, associator, tensor category}

\begin{abstract}
In this paper,  we show that there are infinitely many semisimple tensor (or monoidal) categories of rank two over an algebraically closed field $\mathbb F$.
\end{abstract}

\maketitle
\section*{\bf Introduction}\selabel{1}


In recent years, great breakthroughs have been made in the classification of fusion categories. Ostrik classified the fusion categories of rank 2 \cite{Ostrik03-2}. He proved that there exist only 4 fusion categories of rank 2 up to tensor equivalence. Furthermore, he classified the pivotal fusion categories of rank 3 \cite{Ostrik03-3}. Larson investigated the Pseudo-unitary non-self-dual fusion categories of rank 4 \cite{HKL}. Rowell, Strong and Wang proved that there exist 35 unitary modular categories of rank less than 4 up to ribbon tensor equivalence \cite{ERZ}. Recently, Chen and Zhang showed that a finite rank tensor category $\mathcal{C}$ is uniquely determined  by  three structure invariants \cite{ChenZhang}: the Green ring $r(\mathcal{C})$, the Auslander algebra $A(\mathcal{C})$ and the associator constraints. In particular, if $\mathcal{C}$ is  semisimple tensor category of finite rank, then $\mathcal{C}$ is uniquely determined by the Green ring $r(\mathcal{C})$ and the associator constraints. We ask ourselves the question:  given a $\mathbb{Z}_+$-ring $R$,  is there a semisimple tensor category such that its Green ring is isomorphic to $R$? Such a tensor category is called a categorification of the $\mathbb{Z}_+$-ring. If it exists, then how many tensor categories are there such that their Green rings are isomorphic to $R$?

In this paper, we consider the categorization of the $\mathbb{Z}_+$-rings $R$ with a $\mathbb{Z}_+$-basis $\{r_1=1,r_2\}$ such that $r^2_2\neq 0$.   Let
$r_2^2=mr_1+nr_2$ for some nonnegative integers $m, n$ with $m+n>0$. We divide it into three cases: $m>0$, $n=0$; $m=0$, $n>0$; $m, n>0$.  We show that there is no semisimple tensor category $\mathcal{C}$ such that the Green ring $r(\mathcal{C})$ is isomorphic to $R$ with $m>1$ and $n=0$. If $m=1$ and $n=0$, then there is only two semisimple tensor categories $\mathcal{C}_+$ and $\mathcal{C}_-$ (up to tensor equivalent) when char$(\mathbb{F})\neq 2$. Moreover,  $\mathcal{C}_+ = \mathcal{C}_-$ if char$(\mathbb{F})= 2$. If $m=0$ and $n=1$, then there is only one semisimple tensor category $\mathcal{C}_1$. If $m=0$ and $n>1$, then there are at least two semisimple tensor categories $\mathcal{C}_n$ and $\mathcal{C}'_n$ for any $n$. If $m=n=1$, then there is only two semisimple tensor categories $\mathcal{C}_{(0)}$ and $\mathcal{C}'_{(0)}$ when char$(\mathbb{F})\neq 5$, and there is only one semisimple tensor category $\mathcal{C}_{(5)}$ when char$(\mathbb{F})= 5$. Finally, we proved that if $m>1$ and $2m>n^2$, then there is no semisimple tensor category $\mathcal{C}$ such that $r(\mathcal{C})$ is isomorphic to $R$.

The paper is organized as follows. In Section 1, we introduce some basic notations including a ${\bf (m,s)}$-type matrix and a $\mathbb{Z}_+$-ring. In Section 2, we characterize the structure invariants of the semisimple tensor category of rank two using $6$-$j$ symbols. We prove that the pentagon commutative diagram of a tensor category is equivalent to the Biedenharn-Elliot identity.  In Section 3, we consider the categorizations of $R$ with a $\mathbb{Z}_+$-basis $\{r_1=1,r_2\}$ such that $r^2_2\neq 0$.

\section{\bf Preliminaries and Notations}\selabel{2}

Throughout, let $\mathbb{F}$ be an algebraically closed field, $\mathbb Z$ the ring of  integers, and  $\mathbb{N}$  the set of nonnegative integers. Unless otherwise stated, all algebras are defined over $\mathbb F$, ${\rm dim}$ means ${\rm dim}_{\mathbb F}$.
All rings and algebras are assumed to be associative with identity. The definitions of a $\mathbb{Z}_+$-basis and a $\mathbb{Z}_+$-ring
can be found in \cite{EK, L1, Ostrik03-1}. For the theory of  (tensor) categories, we refer the reader to \cite{BaKir, EGNO, Ka}.

\begin{definition}\label{+ring}
(1)  Let $R$ be a ring free as a module over $\mathbb Z$.  A $\mathbb{Z}_+$-basis of  $R$
is a $\mathbb Z$-basis $B=\{b_i\}_{i\in I}$ such that for any $i, j\in I$,
$b_ib_j=\sum_lc_{ijl}b_l$ with $c_{ijl}\in{\mathbb N}$.\\
(2) A $\mathbb{Z}_+$-ring is a $\mathbb Z$-algebra with identity $1$ endowed with a
fixed $\mathbb{Z}_+$-basis.\\
(3) A $\mathbb{Z}_+$-ring with a $\mathbb{Z}_+$-basis $B$ is a unital
$\mathbb{Z}_+$-ring if $1\in B$.
\end{definition}

For an $\mathbb{F}$-algebra $A$ and two positive integers $m$ and $n$, we denote by $M_{m\times n}(A)$ the $\mathbb F$-space
consisting of all $m\times n$-matrices over $A$, and by $M_m(A):=M_{m\times m}(A)$ the full matrix algebra of $m\times m$-matrices over $A$.
If $B$ is a subspace of $A$, let $M_{m\times n}(B)$ denote the subspace of $M_{m\times n}(A)$ consisting of all
$m\times n$-matrices over $A$ with entries contained in $B$.

In the following, assume that $A$ is the $\mathbb{F}$-algebra $\mathbb{F}^2$ with a given set of orthogonal primitive idempotents
$\{e_1, e_2\}$ such that $e_1+e_2=1$. In this case, $Ae_i=\mathbb{F}e_i$, $i=1,2$

For any $\mathbf{m}=(m_1, m_2)\in {\mathbb N}^2$, let $|\mathbf{m}|:=m_1+m_2$.
For $\mathbf{m}=(m_1,m_2), \mathbf{s}=(s_1,s_2)\in{\mathbb N}^2$ with $|\mathbf m|>0$
and $|\mathbf s|>0$, an $(\mathbf{m}, \mathbf{s})$-type matrix over $A$
is a matrix $X\in M_{|\mathbf{m}|\times|\mathbf{s}|}(A)$ defined by
$X=\left(\begin{array}{cc}
X_1 & 0 \\
0 & X_2 \\
\end{array}\right)$
with $X_{i}\in M_{m_i\times s_i}(\mathbb{F}e_i)$, $i=1,2$.
A $(\mathbf{0}, \mathbf{m})$-type matrix over $A$ means a $1\times|\mathbf{m}|$ zero matrix over $A$,
and an $(\mathbf{m}, \mathbf{0})$-type matrix over $A$ means an $|\mathbf{m}|\times 1$ zero matrix over $A$.
Of course,  a  $(\mathbf{0}, \mathbf{0})$-type matrix over $A$ means a $1\times 1$ zero matrix over $A$.
Here $\mathbf{0}=(0,0)$.

For any positive integer $m$, let $I_m$ denote the $m\times m$ identity matrix
over $A$ (or over $\mathbb F$). Then the $m\times m$ identity matrix over $\mathbb{F}e_i$
is $e_iI_m$ (or $I_me_i$), where $i=1,2$.

For $\mathbf{m}, \mathbf{s}\in{\mathbb N}^2$ with $\mathbf m\neq\mathbf 0$
and $\mathbf s\neq\mathbf 0$, let $M_{\mathbf{m}\times\mathbf{s}}(A)$ be the
$\mathbb{F}$-space consisting of all $(\mathbf{m}, \mathbf{s})$-type matrices over $A$.
Then $M_{\mathbf{m}}(A):=M_{\mathbf{m}\times\mathbf{m}}(A)$ is an associative
$\mathbb{F}$-algebra with the identity $E_{\mathbf m}$ given by
$E_{\mathbf m}=\left(\begin{array}{cc}
e_1I_{m_1} & 0 \\
0 & e_2I_{m_2} \\
\end{array}\right)$,
here we regard $E_{\mathbf m}=e_2I_{m_2}$ if $m_1=0$ and $E_{\mathbf m}=e_1I_{m_1}$ if $m_2=0$.
For convenience, we let $E_{\mathbf 0}:=0$, the $1\times 1$ zero matrix over $A$.
Then $M_{\mathbf{0}}(A)=E_{\mathbf 0}=0$.

Note that $\mathbb{N}^2$ is a monoid with respect to the addition:
$(m_1, m_2)+(s_1, s_2)=(m_1+s_1, m_2+s_2)$.

Let $\mathbf{m}$, $\mathbf{m}'$, $\mathbf{s}\in\mathbb{N}^2$.
For any $X=\left(\begin{array}{cc}
X_1&0\\
0&X_2\\
\end{array}\right)\in M_{\mathbf{m}\times\mathbf{s}}(A)$ with $X_{i}\in M_{m_i\times s_i}(\mathbb{F}e_i)$
and
$Y=\left(\begin{array}{cc}
Y_1&0\\
0&Y_2\\
\end{array}\right)\in M_{\mathbf{m}'\times\mathbf{s}}(A)$ with $Y_{i}\in M_{m'_i\times s_i}(\mathbb{F}e_i)$,
define $X\underline{\oplus}Y\in M_{(\mathbf{m}+\mathbf{m}')\times\mathbf{s}}(A)$ by
$$X\underline{\oplus}Y:=\left(\begin{array}{cc}
X_1&0\\
Y_1&0\\
0&X_2\\
0&Y_2\\
\end{array}\right).$$
It is obvious that $\underline{\oplus}$ is associative, but not commutative in general.

Let $P_{\mathbf{m},\mathbf{m}'}$ be an $(|\mathbf{m}+\mathbf{m}'|)\times(|\mathbf{m}+\mathbf{m}'|)$
permutation matrix defined by
$$P_{\mathbf{m},\mathbf{m}'}=\left(\begin{array}{cccc}
I_{m_1}&0&0&0\\
0&0&I_{m'_1}&0\\
0&I_{m_2}&0&0\\
0&0&0&I_{m'_2}\\
\end{array}\right).$$
It is obvious that $X\underline{\oplus}Y=P_{\mathbf{m},\mathbf{m}'}\left(\begin{array}{c}
X\\
Y\\
\end{array}\right)$.
In a similar way, we define another sum $X\overline{\oplus}Y\in M_{\mathbf{m}\times(\mathbf{s}+\mathbf{s}')}(A)$ for two matrices
$X\in M_{\mathbf{m}\times\mathbf{s}}(A)$ and  $ Y\in M_{\mathbf{m}\times\mathbf{s}'}(A)$ as follows:
$$X\overline{\oplus}Y:=\left(\begin{array}{cccc}
X_1&Y_1&0&0\\
0&0&X_2&Y_2\\
\end{array}\right).$$
Like the sum $\underline{\oplus}$,  the sum $\overline{\oplus}$ is associative, but not commutative in general .
Obviously, $X\overline{\oplus}Y=(X, Y)P_{\mathbf{s}, \mathbf{s}'}^T$, where $P^T$ denotes the transposed
matrix of $P$.

Let $\mathbf{m}_1, \mathbf{m}_2, \cdots, \mathbf{m}_r$, $\mathbf{s}_1, \mathbf{s}_2, \cdots, \mathbf{s}_l\in\mathbb{N}^2$.
For a matrix
$$X=\left(\begin{array}{cccc}
X_{11}&X_{12}&\cdots&X_{1l}\\
X_{21}&X_{22}&\cdots&X_{2l}\\
\cdots&\cdots&\cdots&\cdots\\
X_{r1}&X_{r2}&\cdots&X_{rl}\\
\end{array}\right)$$
over $A$ with $X_{ij}\in M_{\mathbf{m}_i\times\mathbf{s}_j}(A)$, $1\leqslant i\leqslant r$,
$1\leqslant j\leqslant l$, we define an $(\mathbf{m}, \mathbf{s})$-type matrix $\Pi(X)$ over $A$ by setting
$$\Pi(X):=(X_{11}\overline{\oplus}X_{12}\overline{\oplus}\cdots\overline{\oplus}X_{1l})
\underline{\oplus}(X_{21}\overline{\oplus}X_{22}\overline{\oplus}\cdots\overline{\oplus}X_{2l})
\underline{\oplus}\cdots
\underline{\oplus}(X_{r1}\overline{\oplus}X_{r2}\overline{\oplus}\cdots\overline{\oplus}X_{rl}),$$
where $\mathbf{m}=\mathbf{m}_1+\mathbf{m}_2+\cdots+\mathbf{m}_r$ and
$\mathbf{s}=\mathbf{s}_1+\mathbf{s}_2+\cdots+\mathbf{s}_l$. Obviously, we have
$$\Pi(X)=(X_{11}\underline{\oplus}X_{21}\underline{\oplus}\cdots\underline{\oplus}X_{r1})
\overline{\oplus}(X_{12}\underline{\oplus}X_{22}\underline{\oplus}\cdots\underline{\oplus}X_{r2})
\overline{\oplus}\cdots
\overline{\oplus}(X_{1l}\underline{\oplus}X_{2l}\underline{\oplus}\cdots\underline{\oplus}X_{rl}).$$

For $\mathbf{m}_1, \mathbf{m}_2, \cdots, \mathbf{m}_r\in\mathbb{N}^2$, define a permutation matrix
$P_{\mathbf{m}_1,\cdots,\mathbf{m}_r}$ recursively on $r$ as follows:
for $r=1$, $P_{\mathbf{m}_1}=I_{|\mathbf{m}_1|}$; for $r=2$,
$P_{\mathbf{m}_1, \mathbf{m}_2}$ is defined as before; for $r>2$,
$P_{\mathbf{m}_1,\cdots,\mathbf{m}_r}:=P_{\mathbf{m}_1+\cdots+\mathbf{m}_{r-1}, \mathbf{m}_r}
\left(\begin{array}{cc}
P_{\mathbf{m}_1,\cdots,\mathbf{m}_{r-1}}&0\\
0&I_{|\mathbf{m}_r|}\\
\end{array}\right)$.
Then one can see that
$$\Pi(X)=P_{\mathbf{m}_1,\cdots,\mathbf{m}_r}XP_{\mathbf{s}_1,\cdots,\mathbf{s}_l}^T.$$

Let $A$ and $B$ be two algebras over $\mathbb F$. Let $X=(x_{ij})\in M_{m\times s}(A)$ and
$Y=(y_{ij})\in M_{m'\times s'}(B)$. Then one can define a matrix
$X\ot_{\mathbb F}Y\in M_{mm'\times ss'}(A\ot_{\mathbb F}B)$ in a natural way:

$$X\ot_{\mathbb F}Y:=\left(\begin{array}{cccc}
X\ot_{\mathbb F}y_{11}&X\ot_{\mathbb F}y_{12}&\cdots&X\ot_{\mathbb F}y_{1s'}\\
X\ot_{\mathbb F}y_{21}&X\ot_{\mathbb F}y_{22}&\cdots&X\ot_{\mathbb F}y_{2s'}\\
\cdots&\cdots&\cdots&\cdots\\
X\ot_{\mathbb F}y_{m'1}&X\ot_{\mathbb F}y_{m'2}&\cdots&X\ot_{\mathbb F}y_{m's'}\\
\end{array}\right),$$
where
$$X\ot_{\mathbb F}y_{ij}:=\left(\begin{array}{cccc}
x_{11}\ot_{\mathbb F}y_{ij}&x_{12}\ot_{\mathbb F}y_{ij}&\cdots&x_{1s}\ot_{\mathbb F}y_{ij}\\
x_{21}\ot_{\mathbb F}y_{ij}&x_{22}\ot_{\mathbb F}y_{ij}&\cdots&x_{2s}\ot_{\mathbb F}y_{ij}\\
\cdots&\cdots&\cdots&\cdots\\
x_{m1}\ot_{\mathbb F}y_{ij}&x_{m2}\ot_{\mathbb F}y_{ij}&\cdots&x_{ms}\ot_{\mathbb F}y_{ij}\\
\end{array}\right)$$
for any $1\<i\<m'$ and $1\<j\<s'$.
It is obvious that the above tensor product is associative.
If $A=B=\mathbb F$, then we may regard $\mathbb{F}\ot_{\mathbb{F}}\mathbb{F}=\mathbb{F}$ and
$x_{ij}\ot_{\mathbb{F}}y_{i'j'}=x_{ij}y_{i'j'}$, and so $X\ot_{\mathbb{F}}Y\in M_{mm'\times ss'}(\mathbb{F})$.
The following property of the tensor product is easy to check.

\begin{lemma}\label{composition}
Assume that $A$ and $B$ are two $\mathbb F$-algebras. Let $X\in M_{m\times s}(A)$,
$X_1\in M_{s\times t}(A)$, $Y\in M_{m'\times s'}(B)$ and $Y_1\in M_{s'\times t'}(B)$.
Then $(X\ot_{\mathbb F}Y)(X_1\ot_{\mathbb F}Y_1)=(XX_1)\ot_{\mathbb F}(YY_1)$.
\end{lemma}

Let  $n$ be a positive integer. For any $1\<i\<n$, let $\e_{n,i}=(\d_{i1}\ \d_{i2}\ \cdots\ \d_{in})\in M_{1\times n}(\mathbb F)$, where $\d_{ij}=1$ for $j=i$ and $\d_{ij}=0$ for $j\neq i$.
For any $1\<i,j\<n$,
let $I^{[n]}_{i,j}=\e^T_{n,i}\e_{n,j}\in M_n(\mathbb F)$.
Let
$$I_{[n]}=\sum_{i,j=1}^nI_{i,j}^{[n]}\ot_{\mathbb F}I_{j,i}^{[n]}
=\left(\begin{array}{cccc}
I_{1,1}^{[n]}&I_{2,1}^{[n]}&\cdots&I_{n,1}^{[n]}\\
I_{1,2}^{[n]}&I_{2,2}^{[n]}&\cdots&I_{n,2}^{[n]}\\
\cdots&\cdots&\cdots&\cdots\\
I_{1,n}^{[n]}&I_{2,n}^{[n]}&\cdots&I_{n,n}^{[n]}\\
\end{array}\right)\in M_{n^2}(\mathbb F).$$
If there is no ambiguity, $I^{[n]}_{i,j}$ is simply written as $I_{i,j}$.
For any positive integers  $m$ and $n$, the zero matrix 0 in $M_{m\times n}(\mathbb F)$
(or in $M_{n}(\mathbb F)$) is sometimes denoted  $0_{m\times n}$ (or $0_n$)
in order to indicate its order.

\section{\bf Construction of semisimple tensor categories of rank two}\selabel{2}

Throughout this section, assume that $R$ is a $\mathbb{Z}_+$-ring with a unital
$\mathbb{Z}_+$-basis $\{r_1=1, r_2\}_{i\in I}$ such that $r_2^2\neq 0$,
and $A=\mathbb{F}^2$ is a 2-dimensional $\mathbb{F}$-algebra as given in the last section.
Let $\mathbb{F}^{\times}$ be the multiplicative group of nonzero elements of $\mathbb F$.

Now we construct a category $\mathcal C$ as follows:\\
(1) Ob$(\mathcal C):={\mathbb N}^2$;\\
(2) ${\rm Mor}(\mathcal C)$: for ${\mathbf m}=(m_1, m_2)$
and ${\mathbf s}=(s_1, s_2)$ in Ob$(\mathcal C)={\mathbb N}^2$, define
$${\rm Hom}_{\mathcal C}({\mathbf m}, {\mathbf s}):=M_{{\mathbf s}\times{\mathbf m}}(A);$$
(3) Composition: for ${\mathbf s}$, ${\mathbf m}$, ${\mathbf t}\in{\mathbb N}^2$,
define the composition
$${\rm Hom}_{\mathcal C}({\mathbf m}, {\mathbf s})\times {\rm Hom}_{\mathcal C}({\mathbf t}, {\mathbf m})
\ra {\rm Hom}_{\mathcal C}({\mathbf t}, {\mathbf s}),\
(X, Y)\mapsto X\circ Y$$
by $X\circ Y=XY\in M_{{\mathbf s}\times{\mathbf t}}(A)={\rm Hom}_{\mathcal C}({\mathbf t}, {\mathbf s})$, the usual product of matrices.

Let ${\mathbf{e}}_1=(1, 0)$, ${\mathbf{e}}_2=(0, 1)\in{\rm Ob}({\mathcal C})$.
Then by the definition of ${\mathcal C}$,
${\rm Hom}_{\mathcal C}({\mathbf{e}}_1, {\mathbf{e}}_2)=0$, ${\rm Hom}_{\mathcal C}({\mathbf{e}}_2, {\mathbf{e}}_1)=0$,
and ${\rm End}_{\mathcal C}({\mathbf{e}}_i)=e_iAe_i=\mathbb{F}e_i$, $i=1,2$.
By \cite[Proposition 2.14 and Corollary 2.15]{ChenZhang}, we have the following proposition.

\begin{proposition}\label{2.1}
$\mathcal C$ is a semisimple category over $\mathbb F$.
Moreover, ${\mathbf{e}}_1$, ${\mathbf{e}}_2$ are the only two non-isomorphic simple objects of $\mathcal C$.
\end{proposition}

Note that $R$ is a $\mathbb{Z}_+$-ring with the unital $\mathbb{Z}_+$-basis
$\{r_1=1, r_2\}$. We have $r_2^2=mr_1+nr_2$ for some $m, n\in\mathbb{N}$ with $m+n>0$.
Let $\mathbf{c}_{22}=(m, n)$ and $\mathbf{c}_{1i}=\mathbf{c}_{i1}=\mathbf{e}_i$, $i=1, 2$.
Then $\mathbf{c}_{ij}=(c_{ij1}, c_{ij2})\in\mathbb{N}^2$ and $r_ir_j=c_{ij1}r_1+c_{ij2}r_2$, where $1\leqslant i, j\leqslant 2$.

For any $\mathbf{m}=(m_1, m_2)$ and $\mathbf{s}=(s_1, s_2)$ in ${\rm Ob}(\mathcal C)$,
define $\mathbf{m}\ot\mathbf{s}\in{\rm Ob}(\mathcal C)$ by
\begin{eqnarray}
\mathbf{m}\ot\mathbf{s}:=\sum_{i, j=1}^2m_is_j\mathbf{c}_{ij}.\label{T}
\end{eqnarray}
Then $\mathbf{m}\ot\mathbf{s}=\mathbf{0}$ if and only if $\mathbf{m}=\mathbf{0}$ or $\mathbf{s}=\mathbf{0}$.
Moreover, $\mathbf{m}\ot\mathbf{e}_1=\mathbf{e}_1\ot\mathbf{m}=\mathbf{m}$.
By the associativity of the ring $R$, one can see that $(\mathbf{m}\ot\mathbf{s})\ot\mathbf{t}
=\mathbf{m}\ot(\mathbf{s}\ot\mathbf{t})$, denoted  $\mathbf{m}\ot\mathbf{s}\ot\mathbf{t}$,
where $\mathbf{m}, \mathbf{s}, \mathbf{t}\in {\rm Ob}(\mathcal C)$.

Define a vector space $M(R, A)$ over $\mathbb{F}$ by
$$M(R, A):=\oplus_{1\leqslant i,i',j,j'\leqslant 2}M_{\mathbf{c}_{i'j'}\times\mathbf{c}_{ij}}(A).$$
Then $M(R, A)$ is an associative $\mathbb{F}$-algebra with the multiplication
defined as follows: if $X\in M_{\mathbf{c}_{i'j'}\times\mathbf{c}_{ij}}(A)$ and
$Y\in M_{\mathbf{c}_{i''j''}\times\mathbf{c}_{i_1j_1}}(A)$,
then $XY$ is the usual matrix product for $(i, j)=(i'',j'')$,
and $XY=0$ for $(i, j)\neq(i'',j'')$.
The identity of $M(R,A)$ is $(E_{\mathbf{c}_{ij}})_{i,j\in\{1,2\}}$,
where $E_{\mathbf{c}_{ij}}\in M_{\mathbf{c}_{ij}}(A)$ is given in the last section.

Note that $A\otimes_{\mathbb F}A$ has an $\mathbb F$-basis $\{e_i\ot e_j|1\<i,j\<2\}$.
Define an $\mathbb F$-linear map $\phi: A\otimes_{\mathbb F}A\ra M(R, A)$ by
$\phi(e_i\otimes_{\mathbb F}e_j)=E_{\mathbf{c}_{ij}}\in M_{\mathbf{c}_{ij}}(A)$, $1\<i, j\<2$.
Then $\phi$ is the unique algebra map from $A\otimes_{\mathbb F}A$ to $M(R, A)$
such that the two conditions $(\phi$1) and $(\phi$2) in \cite[page 266]{ChenZhang} are satisfied.
Moreover, we have
$\phi(x)\in M_{\mathbf{c}_{i'j'}\times\mathbf{c}_{ij}}(A)$ for any $x\in e_{i'}Ae_{i}\ot_{\mathbb F}e_{j'}Ae_j
\subset A\ot_{\mathbb F}A$, where $1\leqslant i', i, j', j\leqslant 2$.

For any homogeneous matrix $X=(x_{ij})$ over $A\ot_{\mathbb F}A$, $\phi(X)=(\phi(x_{ij}))$
is a well-defined homogeneous matrix over $A$.
Let $X\in{\rm Hom}_{\mathcal C}(\mathbf{m}_1, \mathbf{s}_1)$
and $Y\in{\rm Hom}_{\mathcal C}(\mathbf{m}_2, \mathbf{s}_2)$.
Then one  can also define  $X\ot Y\in {\rm Hom}_{\mathcal C}(\mathbf{m}_1\ot\mathbf{m}_2, \mathbf{s}_1\ot\mathbf{s}_2)$
as follows:
\begin{eqnarray}
X\ot Y:=\Pi(\phi(X\ot_{\mathbb F}Y))=P(\mathbf{s}_1, \mathbf{s}_2)\phi(X\ot_{\mathbb F}Y)P(\mathbf{m}_1, \mathbf{m}_2)^T \label{t}
\end{eqnarray}
for $\mathbf{m}_1\ot\mathbf{m}_2\neq\mathbf{0}$ and $\mathbf{s}_1\ot\mathbf{s}_2\neq\mathbf{0}$,
and $X\ot Y:=0$  otherwise.

Let $\mathbf{0}\neq\mathbf{m}=(m_1, m_2)\in\mathbb{N}^2$. For any $1\<i\<2$ with $m_i>0$, and any $1\leqslant k\leqslant m_i$,
define a matrix $Y^{\mathbf m}_{i,k}=(y_1, y_2, \cdots, y_{|\mathbf m|})\in M_{\mathbf{e}_i\times\mathbf{m}}(A)$ by
$$y_j=\left\{\begin{array}{ll}
e_i,& j=k+\d_{i2}m_1,\\
0,& \text{ otherwise},\\
\end{array}\right.$$
where $\d_{i2}$ is the Kronecker symbol. Let
$X^{\mathbf m}_{i,k}=(Y^{\mathbf m}_{i,k})^T$, the transposed matrix of $Y^{\mathbf m}_{i,k}$.
Then $X^{\mathbf m}_{i,k}\in M_{\mathbf{m}\times\mathbf{e}_i}(A)$. We have
$$Y^{\mathbf m}_{i,k}X^{\mathbf m}_{i',k'}
=\left\{\begin{array}{ll}
e_i=E_{\mathbf{e}_i},& (i,k)=(i',k'),\\
0,& \text{ otherwise},\\
\end{array}\right.$$
and $\sum_{i=1}^2\sum_{1\leqslant k\leqslant m_i}X^{\mathbf m}_{i,k}Y^{\mathbf m}_{i,k}
=\sum_{1\leqslant i\leqslant 2,  m_i>0}\sum_{k=1}^{m_i}X^{\mathbf m}_{i,k}Y^{\mathbf m}_{i,k}=E_{\mathbf m}$.

For any $1\<i, j\<2$, let $a_{1,i,j}=a_{i,1,j}=a_{i,j,1}=E_{\mathbf{e}_i\ot\mathbf{e}_j}\in M_{\mathbf{e}_i\ot\mathbf{e}_j}(A)$.
Note that $\mathbf{e}_1\ot\mathbf{e}_i\ot\mathbf{e}_j=\mathbf{e}_i\ot\mathbf{e}_1\ot\mathbf{e}_j
=\mathbf{e}_i\ot\mathbf{e}_j\ot\mathbf{e}_1=\mathbf{c}_{ij}$
and $\mathbf{e}_2\ot\mathbf{e}_2\ot\mathbf{e}_2=(mn,m+n^2)$.

From now on, we assume that $a_{2,2,2}$ is an invertible element in $M_{\mathbf{e}_2\ot\mathbf{e}_2\ot\mathbf{e}_2}(A)$.

\begin{lemma}\label{2.2}
Let $1\<i,j,l,i',j',l'\<2$, $x\in e_{i'}Ae_{i}$, $y\in e_{j'}Ae_{j}$ and $z\in e_{l'}Ae_{l}$.
Then $(x\ot(y\ot z))a_{i, j, l}=a_{i', j', l'}((x\ot y)\ot z)$.
\end{lemma}

\begin{proof}
If $(i',j',l')\neq(i,j,l)$, then $(x\ot(y\ot z))=((x\ot y)\ot z)=0$ since $e_1Ae_2=e_2Ae_1=0$,
 so $(x\ot(y\ot z))a_{i, j, l}=a_{i', j', l'}((x\ot y)\ot z)$.
If $(i',j',l')=(i,j,l)$, then $x=\a e_i$, $y=\b e_j$ and $z=\g e_l$ for some $\a, \b, \g\in\mathbb F$.
In this case, $(x\ot(y\ot z))=((x\ot y)\ot z)=\a\b\g E_{\mathbf{e}_i\ot\mathbf{e}_j\ot\mathbf{e}_l}$
by \cite[Corollary 3.6]{ChenZhang}. Thus $(x\ot(y\ot z))a_{i, j, l}=a_{i, j, l}((x\ot y)\ot z)$.
\end{proof}

Now let us consider the equality Ass(4) in \cite{ChenZhang}:
\begin{eqnarray}
\nonumber&\sum\limits_{j=1}^2\sum\limits_{1\< k\< c_{i_2i_3j}}
(e_{i_1}\ot  a_{i_2, i_3, i_4}(X^{\mathbf{c}_{i_2i_3}}_{j,k}\ot e_{i_4}))
a_{i_1,j,i_4}((e_{i_1}\ot  Y^{\mathbf{c}_{i_2i_3}}_{j,k})a_{i_1,i_2,i_3}\ot e_{i_4})\\
\nonumber=&\sum\limits_{j,j'=1}^2\sum\limits_{1\< k\< c_{i_3i_4j}}
\sum\limits_{1\< k'\< c_{i_1i_2j'}}
(e_{i_1}\ot (e_{i_2}\ot  X^{\mathbf{c}_{i_3i_4}}_{j,k}))a_{i_1, i_2, j}
(X^{\mathbf{c}_{i_1i_2}}_{j',k'}\ot  Y^{\mathbf{c}_{i_3i_4}}_{j,k})\\
&\hspace{3.2cm}\cdot a_{j',i_3,i_4}
((Y^{\mathbf{c}_{i_1i_2}}_{j',k'}\ot  e_{i_3})\ot  e_{i_4}), \label{Ass4}
\end{eqnarray}
where $1\<i_1, i_2, i_3, i_4\<2$.

\begin{lemma}\label{2.3}
If one of $i_1, i_2, i_3, i_4$ is equal to 1, then the equality (\ref{Ass4}) holds.
\end{lemma}

\begin{proof}
Follows from a straightforward verification.
\end{proof}

When $i_1=i_2=i_3=i_4=2$, the equality (\ref{Ass4}) becomes
\begin{eqnarray}
\nonumber&\sum\limits_{j=1}^2\sum\limits_{1\< k\< c_{22j}}
(e_{2}\ot  a_{2, 2, 2}(X^{\mathbf{c}_{22}}_{j,k}\ot e_{2}))
a_{2,j,2}((e_{2}\ot  Y^{\mathbf{c}_{22}}_{j,k})a_{2,2,2}\ot e_{2})\\
\nonumber=&\sum\limits_{j,j'=1}^2\sum\limits_{1\< k\< c_{22j}}
\sum\limits_{1\< k'\< c_{22j'}}
(e_{2}\ot (e_{2}\ot  X^{\mathbf{c}_{22}}_{j,k}))a_{2, 2, j}
(X^{\mathbf{c}_{22}}_{j',k'}\ot  Y^{\mathbf{c}_{22}}_{j,k})\\
&\hspace{3.2cm}\cdot a_{j',2,2}((Y^{\mathbf{c}_{22}}_{j',k'}\ot e_{2})\ot e_{2}).\label{a222}
\end{eqnarray}

 To simplify the equality (\ref{a222}), we use the $6$-$j$ symbols of semisimple tensor categories.

Let $H^{ij}_k=$ ${\rm Hom}_{\mathcal C}(\mathbf{e}_k,\mathbf{e}_i\otimes\mathbf{e}_j)=M_{\mathbf{c}_{ij}\times \mathbf{e}_k}(A)$ be the multiplicity spaces.
Then ${\rm dim}(H^{ij}_k)=c_{ijk}$, and the decomposition of the tensor product of simple objects can be described as follows:
 $$\mathbf{e}_i\otimes\mathbf{e}_j\cong \oplus_{k=1}^2{\rm dim}(H^{ij}_k)\mathbf{e}_k.$$
Hence we have
$$(\mathbf{e}_i\otimes\mathbf{e}_j)\otimes \mathbf{e}_l\cong \oplus_{s=1}^2\oplus_{k=1}^2{\rm dim}(H^{kl}_s\otimes_{\mathbb{F}} H^{ij}_k)\mathbf{e}_s$$
and
$$\mathbf{e}_i\otimes(\mathbf{e}_j\otimes \mathbf{e}_l)\cong \oplus_{s=1}^2\oplus_{k=1}^2{\rm dim}(H^{ik}_s\otimes_{\mathbb{F}} H^{jl}_k)\mathbf{e}_s.$$

\begin{lemma}\label{2.4}
For any $1\<i,j,k,l,s\<2$, there is a linear injection
$$\s: H^{kl}_s\otimes_{\mathbb{F}} H^{ij}_k\rightarrow{\rm Hom}_{\mathcal C}
(\mathbf{e}_s,(\mathbf{e}_i\otimes\mathbf{e}_j)\otimes \mathbf{e}_l),\
f\otimes_{\mathbb{F}}g\mapsto (g\otimes{\rm id}_{\mathbf{e}_l})f,$$
which induces a linear isomorphism:
$$\oplus_{k=1}^2 H^{kl}_s\otimes_{\mathbb{F}} H^{ij}_k\cong{\rm  Hom}_{\mathcal C}
(\mathbf{e}_s,(\mathbf{e}_i\otimes\mathbf{e}_j)\otimes \mathbf{e}_l).$$
\end{lemma}

\begin{proof}
If $H^{kl}_s\otimes_{\mathbb{F}} H^{ij}_k=0$, then $\s$ is obviously injective.
If $H^{kl}_s\otimes_{\mathbb{F}} H^{ij}_k\neq 0$, then $H^{kl}_s$ has an $\mathbb F$-basis
$\{X^{\mathbf{c}_{kl}}_{s,p}|1\<p\<c_{kls}\}$, and $H^{ij}_k$ has an $\mathbb F$-basis
$\{X^{\mathbf{c}_{ij}}_{k,q}|1\<q\<c_{ijk}\}$.

We first consider the case of $l=1$. In this case,
$(\mathbf{e}_i\otimes\mathbf{e}_j)\otimes \mathbf{e}_l=\mathbf{e}_i\otimes\mathbf{e}_j=\mathbf{c}_{ij}$.
If $s\neq k$, then $H^{k1}_s=0$. Hence $\s$ is injective. If $s=k$, then $c_{k1k}=1$ and
$(X^{\mathbf{c}_{ij}}_{k,q}\otimes{\rm id}_{\mathbf{e}_1})X^{\mathbf{c}_{k1}}_{k,1}
=(X^{\mathbf{c}_{ij}}_{k,q}\ot e_1)e_k=X^{\mathbf{c}_{ij}}_{k,q}$ for all $1\<q\<c_{ijk}$.
Thus, $\s$ is bijective, and  the lemma holds in this case.

Now suppose that $l=2$. First let $i=j=1$. Then $(\mathbf{e}_i\otimes\mathbf{e}_j)\otimes \mathbf{e}_l=\mathbf{e}_2$.
In this case, if $k=2$ or $s=k=1$, then $H^{kl}_s\otimes_{\mathbb{F}} H^{ij}_k=0$; if $k=1$ and $s=2$,
then $c_{122}=c_{111}=1$ and $(X^{\mathbf{c}_{11}}_{1,1}\otimes{\rm id}_{\mathbf{e}_2})X^{\mathbf{c}_{12}}_{2,1}
=(e_1\ot e_2)e_2=X^{\mathbf{c}_{12}}_{2,1}$.
Thus, the lemma holds for $l=2$ and $i=j=1$.
Next, let $i\neq j$. Then $\mathbf{e}_i\otimes\mathbf{e}_j=\mathbf{c}_{ij}=\mathbf{e}_2$ and
$(\mathbf{e}_i\otimes\mathbf{e}_j)\otimes \mathbf{e}_l=\mathbf{e}_2\otimes\mathbf{e}_2=\mathbf{c}_{22}$.
Note that $c_{221}=m$ and $c_{222}=n$.
If $k=1$, then $H^{kl}_s\otimes_{\mathbb{F}} H^{ij}_k=0$.
If $k=2$ and $s=1$, then by \cite[Corollary 3.6]{ChenZhang},
$(X^{\mathbf{c}_{ij}}_{2,1}\otimes{\rm id}_{\mathbf{e}_2})X^{\mathbf{c}_{22}}_{1,p}
=(e_2\ot e_2)X^{\mathbf{c}_{22}}_{1,p}=E_{\mathbf{c}_{22}}X^{\mathbf{c}_{22}}_{1,p}=X^{\mathbf{c}_{22}}_{1,p}$ for all $1\<p\<m$.
If $s=k=2$, then similarly,
$(X^{\mathbf{c}_{ij}}_{2,1}\otimes{\rm id}_{\mathbf{e}_2})X^{\mathbf{c}_{22}}_{2,p}
=X^{\mathbf{c}_{22}}_{2,p}$ for all $1\<p\<n$. Thus, the lemma holds for $l=2$ and $i\neq j$.
Finally, let $i=j=2$. Then $\mathbf{e}_i\otimes\mathbf{e}_j=\mathbf{c}_{22}$ and
$\mathbf{e}_i\otimes\mathbf{e}_j\otimes \mathbf{e}_l=\mathbf{e}_2\otimes\mathbf{e}_2\otimes \mathbf{e}_2=mn\mathbf{e}_1\oplus(m+n^2)\mathbf{e}_2$.
If $s=k=1$, then $H^{kl}_s\otimes_{\mathbb{F}} H^{ij}_k=0$.
If $s=1$ and $k=2$, then $(X^{\mathbf{c}_{22}}_{2,q}\ot{\rm id}_{\mathbf{e}_2})X^{\mathbf{c}_{22}}_{1,p}
=X^{\mathbf{e}_2\otimes\mathbf{e}_2\otimes \mathbf{e}_2}_{1,m(q-1)+p}$
for all $1\<p\<m$ and $1\<q\<n$. This shows the lemma for $s=1$ and $l=i=j=2$.
If $s=2$ and $k=1$, then
$(X^{\mathbf{c}_{22}}_{1,q}\ot{\rm id}_{\mathbf{e}_2})X^{\mathbf{c}_{12}}_{2,1}
=X^{\mathbf{e}_2\otimes\mathbf{e}_2\otimes \mathbf{e}_2}_{2,q}$ for all $1\<q\<m$.
If $s=2$ and $k=2$, then
$(X^{\mathbf{c}_{22}}_{2,q}\ot{\rm id}_{\mathbf{e}_2})X^{\mathbf{c}_{22}}_{2,p}
=X^{\mathbf{e}_2\otimes\mathbf{e}_2\otimes \mathbf{e}_2}_{2,m+n(q-1)+p}$ for all $1\<p, q\<n$.
Thus, the lemma holds for $l=i=j=s=2$. This completes the proof.
\end{proof}

\begin{lemma}\label{2.5}
For any $1\<i,j,k,l,s\<2$, there is a linear injection
$$H^{ik}_s\otimes_{\mathbb{F}} H^{jl}_k\rightarrow {\rm Hom}_{\mathcal C}
(\mathbf{e}_s,\mathbf{e}_i\otimes(\mathbf{e}_j\otimes \mathbf{e}_l)),\
f\ot_{\mathbb F}g\mapsto({\rm id}_{\mathbf{e}_i}\otimes g)f$$
which induces a linear isomorphism
$$\oplus_{k=1}^2 H^{ik}_s\otimes_{\mathbb{F}} H^{jl}_k\cong{\rm Hom}_{\mathcal C}
(\mathbf{e}_s,\mathbf{e}_i\otimes(\mathbf{e}_j\otimes \mathbf{e}_l)).$$
\end{lemma}

\begin{proof}
Similar to the proof of Lemma \ref{2.4}.
\end{proof}

Let $1\<i,j,l,s\<2$. Then the associator
$a_{i,j,l}: (\mathbf{e}_i\otimes\mathbf{e}_j)\otimes \mathbf{e}_l\ra\mathbf{e}_i\otimes(\mathbf{e}_j\otimes \mathbf{e}_l)$
induces a $k$-linear isomorphism
$$\phi^{l,s}_{i,j}: {\rm Hom}_{\mathcal C}(\mathbf{e}_s, (\mathbf{e}_i\otimes\mathbf{e}_j)\otimes \mathbf{e}_l))
\ra {\rm Hom}_{\mathcal C}(\mathbf{e}_s,\mathbf{e}_i\otimes(\mathbf{e}_j\otimes \mathbf{e}_l)),
f\mapsto a_{i,j,l}f.$$
We may regard the isomorphisms in Lemmas \ref{2.4} and \ref{2.5} as identity maps. Then the above isomorphism can be regarded as the following isomorphism
$$\phi^{l,s}_{i,j}: \oplus_{k=1}^{2}  H^{kl}_s\otimes_{\mathbb{F}} H^{ij}_k\rightarrow \oplus_{t=1}^{2} H^{it}_s\otimes_{\mathbb{F}} H^{jl}_t.$$
In this case, the $6$-$j$ symbols are defined to be the component maps of $\phi^{l,s}_{i,j}$:
$$\left\{\begin{array}{ccc}
      i & j & k \\
      l & s & t
    \end{array}\right\}: H^{kl}_s\otimes_{\mathbb{F}} H^{ij}_k\rightarrow H^{it}_s\otimes_{\mathbb{F}} H^{jl}_t,$$
where $1\<k,t\<2$. Note that $\{X^{\mathbf{c}_{kl}}_{s,p}|1\<p\<c_{kls}\}$, $\{X^{\mathbf{c}_{ij}}_{k,q}|1\<q\<c_{ijk}\}$,
$\{X^{\mathbf{c}_{it}}_{s,d}|1\<d\<c_{its}\}$ and $\{X^{\mathbf{c}_{jl}}_{t,r}|1\<r\<c_{jlt}\}$ are the bases of $H^{kl}_s$,
$H^{ij}_k$, $H^{it}_s$ and $H^{jl}_t$, respectively.
Under the identification described above, one can identify $X^{\mathbf{c}_{kl}}_{s,p} \otimes_{\mathbb{F}} X^{\mathbf{c}_{ij}}_{k,q}$
with $X^{\mathbf{e}_i\otimes\mathbf{e}_j\otimes \mathbf{e}_l}_{s,u}$,
and $X^{\mathbf{c}_{it}}_{s,d} \otimes_{\mathbb{F}}X^{\mathbf{c}_{jl}}_{t,r}$ with $X^{\mathbf{e}_i\otimes\mathbf{e}_j\otimes \mathbf{e}_l}_{s,v}$,
respectively, where $u=\sum_{1\<w\<k-1}c_{wls}c_{ijw}+(q-1)c_{kls}+p$ and $v=\sum_{1\<w\<t-1}c_{iws}c_{jlw}+(r-1)c_{its}+d$.

Let $B=\{B^1, B^2\}$ and $C=\{C^1, C^2\}$ be the (ordered) bases of $\oplus_{k=1}^{2}  H^{kl}_s\otimes_{\mathbb{F}} H^{ij}_k$
and $\oplus_{t=1}^{2} H^{it}_s\otimes_{\mathbb{F}} H^{jl}_t$ given by
$$B^k=\left\{\begin{array}{c}
X^{\mathbf{c}_{kl}}_{s,1}\ot_{\mathbb F}X^{\mathbf{c}_{ij}}_{k,1}, X^{\mathbf{c}_{kl}}_{s,2}\ot_{\mathbb F}X^{\mathbf{c}_{ij}}_{k,1},
\cdots, X^{\mathbf{c}_{kl}}_{s,c_{kls}}\ot_{\mathbb F}X^{\mathbf{c}_{ij}}_{k,1},\\
X^{\mathbf{c}_{kl}}_{s,1}\ot_{\mathbb F}X^{\mathbf{c}_{ij}}_{k,2}, X^{\mathbf{c}_{kl}}_{s,2}\ot_{\mathbb F}X^{\mathbf{c}_{ij}}_{k,2},
\cdots, X^{\mathbf{c}_{kl}}_{s,c_{kls}}\ot_{\mathbb F}X^{\mathbf{c}_{ij}}_{k,2},\\
\cdots\cdots\cdots, \\
X^{\mathbf{c}_{kl}}_{s,1}\ot_{\mathbb F}X^{\mathbf{c}_{ij}}_{k,c_{ijk}}, X^{\mathbf{c}_{kl}}_{s,2}\ot_{\mathbb F}X^{\mathbf{c}_{ij}}_{k,c_{ijk}},
\cdots, X^{\mathbf{c}_{kl}}_{s,c_{kls}}\ot_{\mathbb F}X^{\mathbf{c}_{ij}}_{k,c_{ijk}}\\
\end{array}\right\}, \ k=1, 2,$$
and
$$C^t=\left\{\begin{array}{c}
X^{\mathbf{c}_{it}}_{s,1}\ot_{\mathbb F}X^{\mathbf{c}_{jl}}_{t,1}, X^{\mathbf{c}_{it}}_{s,2}\ot_{\mathbb F}X^{\mathbf{c}_{jl}}_{t,1},
\cdots, X^{\mathbf{c}_{it}}_{s,c_{its}}\ot_{\mathbb F}X^{\mathbf{c}_{jl}}_{t,1},\\
X^{\mathbf{c}_{it}}_{s,1}\ot_{\mathbb F}X^{\mathbf{c}_{jl}}_{t,2}, X^{\mathbf{c}_{it}}_{s,2}\ot_{\mathbb F}X^{\mathbf{c}_{jl}}_{t,2},
\cdots, X^{\mathbf{c}_{it}}_{s,c_{its}}\ot_{\mathbb F}X^{\mathbf{c}_{jl}}_{t,2},\\
\cdots\cdots\cdots,\\
X^{\mathbf{c}_{it}}_{s,1}\ot_{\mathbb F}X^{\mathbf{c}_{jl}}_{t,c_{jlt}}, X^{\mathbf{c}_{it}}_{s,2}\ot_{\mathbb F}X^{\mathbf{c}_{jl}}_{t,c_{jlt}},
\cdots, X^{\mathbf{c}_{it}}_{s,c_{its}}\ot_{\mathbb F}X^{\mathbf{c}_{jl}}_{t,c_{jlt}}\\
\end{array}\right\},\ t=1,2.$$
Under the identification as before, both $B$ and $C$ are the canonical bases of
$${\rm Hom}_{\mathcal C}(\mathbf{e}_s,\mathbf{e}_i\ot\mathbf{e}_j\ot\mathbf{e}_l)
=M_{\mathbf{e}_i\ot\mathbf{e}_j\ot\mathbf{e}_l\times \mathbf{e}_s}(A).$$
With respect to two bases $B$ and $C$, the matrix of the linear isomorphism $\phi^{l,s}_{i,j}$ is $A_s$, i.e.,
\begin{eqnarray}
\nonumber&\phi^{l,s}_{i,j}(X^{\mathbf{c}_{1l}}_{s,1}\ot_{\mathbb F}X^{\mathbf{c}_{ij}}_{1,1},
\cdots, X^{\mathbf{c}_{2l}}_{s,c_{2ls}}\ot_{\mathbb F}X^{\mathbf{c}_{ij}}_{2,c_{ij2}})\\
  =&(X^{\mathbf{c}_{i1}}_{s,1}\ot_{\mathbb F}X^{\mathbf{c}_{jl}}_{1,1},
  \cdots, X^{\mathbf{c}_{i2}}_{s,c_{i2s}}\ot_{\mathbb F}X^{\mathbf{c}_{jl}}_{2,c_{jl2}})A_s, \label{relatedmatrix}
\end{eqnarray}
where $A_s$ is a square matrix over $\mathbb{F}e_s$ such that
$a_{i,j,l}=\left(\begin{array}{cc}
A_1&0\\
0&A_2\\
\end{array}\right)$.
In the following, we call $B^k$ and $C^t$ the standard bases of vector space $H^{kl}_s\otimes_{\mathbb{F}} H^{ij}_k$ and  $H^{it}_s\otimes_{\mathbb{F}} H^{jl}_t$, respectively.

\begin{proposition}\label{2.6} The commutative pentagon diagram is equivalent to the Biedenharn-Elliot identity.
That is, the equality (\ref{Ass4}) is equivalent to
$$\begin{array}{l}
\sum\limits_{j=1}^2({\rm id}^{j_1j_7}_{j_0}\ot_{\mathbb{F}}\left\{\begin{array}{ccc}
                                                   j_2& j_3 & j \\
                                                   j_4 & j_7 & j_8
                                                 \end{array}\right\})
 (\left\{\begin{array}{ccc}
 j_1& j & j_6 \\
 j_4 & j_0 & j_7
 \end{array}\right\}\ot_{\mathbb{F}}{\rm id}^{j_2j_3}_{j} )({\rm id}^{j_6j_4}_{j_0}\ot_{\mathbb{F}} \left\{\begin{array}{ccc}
                                                   j_1& j_2 & j_5 \\
                                                   j_3 & j_6 & j
                                                 \end{array}\right\})\\
 =(\left\{\begin{array}{ccc}
 j_1& j_2 & j_5 \\
 j_8 & j_0 & j_7
 \end{array}\right\}\ot_{\mathbb{F}}{\rm id}^{j_3j_4}_{j_8} )P_{23}(\left\{\begin{array}{ccc}
 j_5& j_3 & j_6 \\
 j_4 & j_0 & j_8
 \end{array}\right\}\ot_{\mathbb{F}}{\rm id}^{j_1j_2}_{j_5})\\
 \end{array}$$
for all $j_i=1,2$, $0\<i\<8$, where ${\rm id}^{jl}_k$ denotes the identity map on $H^{jl}_k$,
$P_{23}$ denotes the flip map of the second and the third tensor factors.
Both sides of this equality are linear maps from
$H^{j_6j_4}_{j_0}\otimes_{\mathbb{F}} H^{j_5,j_3}_{j_6}\otimes_{\mathbb{F}} H^{j_1,j_2}_{j_5}$
to $H^{j_1,j_7}_{j_0}\otimes_{\mathbb{F}} H^{j_2,j_8}_{j_7}\otimes_{\mathbb{F}} H^{j_3,j_4}_{j_8}$.
\end{proposition}
\begin{proof}
The equation (\ref{Ass4}) means that the diagram
$$\begin{tikzpicture}[scale=1.0]
\path (0,0) node(a1) {$((\mathbf{e}_{j_1}\otimes \mathbf{e}_{j_2})\otimes \mathbf{e}_{j_3})\otimes \mathbf{e}_{j_4}$}
 (6.8,0) node(b1) {$(\mathbf{e}_{j_1}\otimes \mathbf{e}_{j_2})\otimes (\mathbf{e}_{j_3}\otimes \mathbf{e}_{j_4})$} ;
\path  (0,-1.3) node(a2) {$  (\mathbf{e}_{j_1}\otimes (\mathbf{e}_{j_2}\otimes \mathbf{e}_{j_3}))\otimes \mathbf{e}_{j_4}$} ;
\path(0,-2.6)node(a3){$\mathbf{e}_{j_1}\otimes ((\mathbf{e}_{j_2}\otimes \mathbf{e}_{j_3})\otimes \mathbf{e}_{j_4})$}
(6.8,-2.6) node(b2) {$\mathbf{e}_{j_1}\otimes (\mathbf{e}_{j_2}\otimes (\mathbf{e}_{j_3}\otimes \mathbf{e}_{j_4}))$};
\path (-1.1, -0.6) node(.) {$a_{j_1, j_2, j_3}\ot{\rm id}_{j_4}$}
(-1.1,-1.95)node(.){$a_{\mathbf{e}_{j_1},\mathbf{c}_{j_2j_3}, \mathbf{e}_{j_4}}$}
(3.3, 0.2) node(.) {$a_{\mathbf{c}_{j_1j_2}, \mathbf{e}_{j_3}, \mathbf{e}_{j_4}}$}
(3.3, -2.85) node(.) {${\rm id}_{j_1}\ot a_{\mathbf{e}_{j_2},\mathbf{e}_{j_3},\mathbf{e}_{j_4}}$}
 (5.8, -1.3) node(.) {$a_{\mathbf{e}_{j_1},\mathbf{e}_{j_2},\mathbf{c}_{j_3j_4}}$};
\draw[->] (a1) --(b1);
\draw[->] (b1) --(b2);
\draw[->] (a3) --(b2);
\draw[->] (a1) --(a2);
\draw[->](a2)--(a3);
\end{tikzpicture}$$
commutes for any $j_1, j_2, j_3, j_4\in\{1, 2\}$, where ${\rm id}_j={\rm id}_{\mathbf{e}_j}=e_j$.
Since $\mathcal C$ is a semisimple category,
the above commutative diagram is equivalent to
the following commutative diagram
$$\begin{tikzpicture}[scale=1.0]
\path (0,0) node(a1) {${\rm Hom}_{\mathcal C}
(\mathbf{e}_{j_0},((\mathbf{e}_{j_1}\ot\mathbf{e}_{j_2})\ot\mathbf{e}_{j_3})\ot\mathbf{e}_{j_4})$}
(7,0) node(b1) {${\rm Hom}_{\mathcal C}
	(\mathbf{e}_{j_0}, (\mathbf{e}_{j_1}\ot\mathbf{e}_{j_2})\ot(\mathbf{e}_{j_3}\ot\mathbf{e}_{j_4}))$} ;
\path  (0,-1.3) node(a2) {${\rm Hom}_{\mathcal C}
	(\mathbf{e}_{j_0},  (\mathbf{e}_{j_1}\ot(\mathbf{e}_{j_2}\ot\mathbf{e}_{j_3}))\ot\mathbf{e}_{j_4})$} ;
\path(0,-2.6)node(a3){${\rm Hom}_{\mathcal C}
	(\mathbf{e}_{j_0}, \mathbf{e}_{j_1}\ot((\mathbf{e}_{j_2}\ot\mathbf{e}_{j_3})\ot\mathbf{e}_{j_4}))$}
(7,-2.6) node(b2) {${\rm Hom}_{\mathcal C}
	(\mathbf{e}_{j_0}, \mathbf{e}_{j_1}\ot(\mathbf{e}_{j_2}\ot(\mathbf{e}_{j_3}\ot\mathbf{e}_{j_4})))$};
\path(3.5, 0.2) node(.) {$a_*$}
 (-0.9, -0.6) node(.) {$(a\ot e_{j_4})_*$}
(6.65, -1.3) node(.) {$a_*$}
(-0.4,-1.95)node(.){$a_*$}
(3.5, -2.85) node(.) {$(e_1\ot a)_*$};
\draw[->] (a1) --(b1);
\draw[->] (b1) --(b2);
\draw[->] (a3) --(b2);
\draw[->] (a1) --(a2);
\draw[->](a2)--(a3);
\end{tikzpicture}$$
where $j_0\in \{1,2\}$, and we omit the subscripts of $a$.
Then by Lemmas \ref{2.4}-\ref{2.5} and the above identification,
we conclude that the last commutative diagram is equivalent to the Biedenharn-Elliot identity.
\end{proof}

By Proposition \ref{2.6}, the equation (\ref{a222}) is equivalent to
\begin{eqnarray}
\nonumber&\sum\limits_{j=1}^2({\rm id}^{2j_7}_{j_0}\ot_{\mathbb{F}} \left\{\begin{array}{ccc}
                                                   2& 2 & j \\
                                                   2 & j_7 & j_8
                                                 \end{array}\right\})
 (\left\{\begin{array}{ccc}
 2& j & j_6 \\
 2 & j_0 & j_7
 \end{array}\right\}\ot_{\mathbb{F}}{\rm id}^{22}_{j} )({\rm id}^{j_62}_{j_0}\ot_{\mathbb{F}} \left\{\begin{array}{ccc}
                                                   2& 2 & j_5 \\
                                                   2 & j_6 & j
                                                 \end{array}\right\})\\
=&(\left\{\begin{array}{ccc}
 2& 2 & j_5 \\
 j_8 & j_0 & j_7
 \end{array}\right\}\ot_{\mathbb{F}}{\rm id}^{22}_{j_8} )P_{23}(\left\{\begin{array}{ccc}
 j_5& 2 & j_6 \\
 2 & j_0 & j_8
 \end{array}\right\}\ot_{\mathbb{F}}{\rm id}^{22}_{j_5}), \hspace{3.5cm}\label{a222+1}
\end{eqnarray}
where $j_0, j_5, j_6, j_7, j_8\in\{1, 2\}$.
\section{\bf Categorization of $\mathbb{Z}_+$-ring}
In this section, we study the categorization of $\mathbb{Z}_+$-ring $R$ as stated in last section.

Let $a_{2,2,2}=\left(\begin{array}{cc}
\Lambda_1e_1&0\\
0&\Lambda_2e_2\\
\end{array}\right)$ for some  matrices $\Lambda_1\in M_{mn}(\mathbb F)$, $\Lambda_2\in M_{m+n^2}(\mathbb F)$, and  $\L_2=\left(\begin{array}{cc}
\L_{22}&\L_{23}\\
\L_{32}&\L_{33}\\
\end{array}\right)$
with $\L_{22}\in M_m(\mathbb F)$ and $\L_{33}\in M_{n^2}(\mathbb F)$.
Since $a_{1,i,j}=a_{i,1,j}=a_{i,j,1}=E_{\mathbf{e}_i\ot\mathbf{e}_j}$,
 with respect to the standard bases $B$ and $C$, one can regard the above $6$-$j$ symbols as matrices as follow:
$$\begin{array}{l}
\vspace{0.2cm}
\left\{\begin{array}{ccc}
                                                                       1 & 1 & 1 \\
                                                                       1 & 1 & 1
                                                                     \end{array}\right
 \}=\left\{\begin{array}{ccc}
                                                                       1 & 1 & 1 \\
                                                                       2 & 2 & 2
                                                                     \end{array}\right\}=\left\{\begin{array}{ccc}
                                                                       1 & 2 & 2 \\
                                                                       1 & 2 & 2
                                                                     \end{array}\right\}=\left\{\begin{array}{ccc}
                                                                       2 & 1 & 2 \\
                                                                       1 & 2 & 1
                                                                     \end{array}\right\}=I_1,\\
  \vspace{0.2cm}
\left\{\begin{array}{ccc}
                                                                       1 & 2 & 2 \\
                                                                       2 & 1 & 1
                                                                     \end{array}\right
 \}=\left\{\begin{array}{ccc}
                                                                       2 & 1 & 2 \\
                                                                       2 & 1 & 2
                                                                     \end{array}\right\}=\left\{\begin{array}{ccc}
                                                                       2 & 2 & 1 \\
                                                                       1 & 1 & 2
                                                                     \end{array}\right\}=I_m,\\
 \vspace{0.2cm}
\left\{\begin{array}{ccc}
                                                                       1 & 2 & 2 \\
                                                                       2 & 2 & 2
                                                                     \end{array}\right
 \}=\left\{\begin{array}{ccc}
                                                                       2 & 1 & 2 \\
                                                                       2 & 2 & 2
                                                                     \end{array}\right\}=\left\{\begin{array}{ccc}
                                                                       2 & 2 & 2 \\
                                                                       1 & 2 & 2
                                                                     \end{array}\right\}=I_n,\\
 \vspace{0.2cm}
\left\{\begin{array}{ccc}
      2 & 2 & 2 \\
      2 & 1 & 2
    \end{array}\right
 \}=\L_1, \
\left\{\begin{array}{ccc}
      2 & 2 & 1 \\
      2 & 2 & 1
    \end{array}\right
 \}=\L_{22}, \
\left\{\begin{array}{ccc}
      2 & 2 & 2 \\
      2 & 2 & 1
    \end{array}\right
 \}=\L_{23},\\
 \vspace{0.2cm}
\left\{\begin{array}{ccc}
      2 & 2 & 1 \\
      2 & 2 & 2
    \end{array}\right
 \}=\L_{32}, \
\left\{\begin{array}{ccc}
      2 & 2 & 2 \\
      2 & 2 & 2
    \end{array}\right
 \}=\L_{33}\\
\end{array}$$
and the rest
$\left\{\begin{array}{ccc}
i & j & k \\
l & m & n
\end{array}\right\}=0$.

In the following, we regard each component map of  Equation (\ref{a222+1}) as  a matrix  with respect to the standard bases.
Since $P_{23}$ is a flip map from $H^{j_{5}j_{8}}_{j_0}\otimes H^{j_{3}j_{4}}_{j_8}\otimes H^{j_{1}j_{2}}_{j_5} \longrightarrow H^{j_{5}j_{8}}_{j_0}\otimes H^{j_{1}j_{2}}_{j_5}\otimes H^{j_{3}j_{4}}_{j_8}$, $P_{23}=I_{a}\otimes_{\mathbb{F}} I_{[c,b]}$, where $a=c_{j_5j_8j_0}$, $b=c_{j_3j_4j_8}$, $c=c_{j_1j_2j_5}$, $I^{[c,b]}_{i,j}=\e^T_{c,i}\e_{b,j}\in M_{c\times b}(\mathbb F)$ and
$$I_{[c,b]}=\left(\begin{array}{cccc}
I_{1,1}^{[c,b]}&I_{2,1}^{[c,b]}&\cdots&I_{c,1}^{[c,b]}\\
I_{1,2}^{[c,b]}&I_{2,2}^{[c,b]}&\cdots&I_{c,2}^{[c,b]}\\
\cdots&\cdots&\cdots&\cdots\\
I_{1,b}^{[c,b]}&I_{2,b}^{[c,b]}&\cdots&I_{c,b}^{[c,b]}\\
\end{array}\right)\in M_{bc\times bc}(\mathbb F).$$
In particularly, if $b=c=n$ then $I_{[c,b]}=I_{[n]}$.

 In order to make $\mathcal C$ into a tensor category, we need to investigate when the equation (\ref{a222+1})  holds.
We will consider it in three cases: (1) $m>0$ and $n=0$, (2) $m=0$ and $n>0$,
(3) $m>0$ and $n>0$, respectively.

\subsection{$m>0$ and $n=0$}
 In this case, $\mathbf{e}_2\ot\mathbf{e}_2=\mathbf{c}_{22}=(m,0)=m\mathbf{e}_1$
and $\mathbf{e}_2\ot\mathbf{e}_2\ot\mathbf{e}_2=(0,m)=m\mathbf{e}_2$. Moreover, $\left\{\begin{array}{ccc}
                                                                       1 & 2 & 2 \\
                                                                       2 & 2 & 2
                                                                     \end{array}
 \right\}=\left\{\begin{array}{ccc}
                                                                       2 & 1 & 2 \\
                                                                       2 & 2 & 2
                                                                     \end{array}\right\}=\left\{\begin{array}{ccc}
                                                                       2 & 2 & 2 \\
                                                                       1 & 2 & 2
                                                                     \end{array}\right\}=\left\{\begin{array}{ccc}
      2 & 2 & 2 \\
      2 & 1 & 2
    \end{array}\right
 \}=\left\{\begin{array}{ccc}
      2 & 2 & 2 \\
      2 & 2 & 2
    \end{array}\right
 \}=\left\{\begin{array}{ccc}
      2 & 2 & 2 \\
      2 & 2 & 1
    \end{array}
 \right\}=\left\{\begin{array}{ccc}
      2 & 2 & 1 \\
      2 & 2 & 2
    \end{array}\right
 \}=0$,
$a_{2,2,2}=\L_{22}e_2$ and $\L_{22}=\left\{\begin{array}{ccc}
               2 & 2 & 1 \\
               2 & 2 & 1
             \end{array}\right
\}$.
 Hence the equation (\ref{a222+1}) is equivalent to
$$I_m\ot_{\mathbb{F}} \L^2_{22}=P_{23}.$$

 Let $a_{222}=(\a_{ij})e_2$ and  $\b_{ij}=\sum_{k=1}^m\a_{ik}\a_{kj}$ for all
$1\<i,j\<m$. That is, $(\b_{ij})=\Lambda^2_{22}\in M_m(\mathbb F)$.
Thus, the above equation becomes
$$
  \left(\begin{array}{cccc}
    \b_{11}I_m & \b_{12}I_m & \cdots & \b_{1m}I_m \\
    \b_{21}I_m & \b_{22}I_m & \cdots & \b_{2m}I_m \\
    \cdots & \cdots & \cdots & \cdots \\
    \b_{m1}I_m & \b_{m2}I_m & \cdots & \b_{mm}I_m \\
  \end{array}\right) =\left( \begin{array}{cccc}
I_{1,1}&I_{2,1}&\cdots&I_{m,1}\\
I_{1,2}&I_{2,2}&\cdots&I_{m,2}\\
\cdots&\cdots&\cdots&\cdots\\
I_{1,m}&I_{2,m}&\cdots&I_{m,m}\\
\end{array}\right), $$
where $I_{i,j}=I^{[m]}_{i,j}$. If $m>1$, it is easy to see that there is no invertible matrix $(\b_{ij})$ that
satisfies the above equation.
Now assume $m=1$. Then $a_{2,2,2}=\a e_2$ for some $\a\in\mathbb{F}^{\times}$.
In this case, it follows from the above computation that the equation (\ref{a222+1}) holds
if and only if $\a^2=1$, which is equivalent to that $a_{2,2,2}=e_2$ or $a_{2,2,2}=-e_2$.

Let $a_{2,2,2}=e_2$. Then by Lemmas \ref{2.2}, \ref{2.3} and the above discussion, the family
$\{a_{i,j,l}\in M_{\mathbf{e}_i\ot\mathbf{e}_j\ot\mathbf{e}_l}(A)|1\<i,j,l\<2\}$
satisfies the conditions Ass (1)-(4). For any $\mathbf{m}=(m_1,m_2), \mathbf{s}=(s_1,s_2), \mathbf{t}=(t_1,t_2)\in{\rm Ob}(\mathcal C)$,
define a matrix $a_{\mathbf{m}, \mathbf{s}, \mathbf{t}}\in M_{\mathbf{m}\ot\mathbf{s}\ot\mathbf{t}}(A)$ by
$a_{\mathbf{m}, \mathbf{s}, \mathbf{t}}=E_{\mathbf{0}}=0$ when $\mathbf{m}\ot\mathbf{s}\ot\mathbf{t}=\mathbf{0}$,
and
$$\begin{array}{l}
\hspace{0.5cm}a_{\mathbf{m}, \mathbf{s}, \mathbf{t}}\\
=\!\sum\limits_{i,j,l=1}^2
\!\sum\limits_{1\<k_1\<m_i}\!\sum\limits_{1\<k_2\<s_j}\!\sum\limits_{1\<k_3\<t_l}
\!(X^{\mathbf m}_{i,k_1}\ot (X^{\mathbf s}_{j,k_2}\ot  X^{\mathbf t}_{l,k_3}))a_{i,j,l}
((Y^{\mathbf m}_{i,k_1}\ot Y^{\mathbf s}_{j,k_2})\ot Y^{\mathbf t}_{l,k_3})\\
\end{array}$$
when $\mathbf{m}\ot\mathbf{s}\ot\mathbf{t}\neq\mathbf{0}$.
Thus, by \cite[Theorem 3.16]{ChenZhang}, we have the following proposition.

\begin{lemma}\label{2.7}
Assume $m=1$. Then with the tensor products $(\ref{T})$ and $(\ref{t})$,
$\mathcal C$ is a tensor category over $\mathbb F$, where $\mathbf{e}_1$ is the unit object,
the associativity constraint $a$ is induced by $\{a_{i,j,l}|1\<i,j,l\<2\}$ as above,
the left and right unit constraints are the identities. Denote the tensor category by $\mathcal{C}_+$.
\end{lemma}

Let $a'_{2,2,2}=-e_2$ and $a'_{1,i,j}=a'_{i,1,j}=a'_{i,j,1}=E_{\mathbf{e}_i\ot\mathbf{e}_j}\in M_{\mathbf{e}_i\ot\mathbf{e}_j}(A)$
for any $1\<i,j\<2$. Then similarly, we have the following proposition.

\begin{lemma}\label{2.8}
Assume $m=1$. Then with the tensor products $(\ref{T})$ and $(\ref{t})$,
$\mathcal C$ is a tensor category over $\mathbb F$, where $\mathbf{e}_1$ is the unit object,
the associativity constraint $a'$ is induced by $\{a'_{i,j,l}|1\<i,j,l\<2\}$ similarly as above,
the left and right unit constraints are the identities. Denote the tensor category by $\mathcal{C}_-$.
\end{lemma}

Note that $\mathcal{C}_+=\mathcal{C}_-$ in case char$(\mathbb F)=2$. In case char$(\mathbb F)\neq 2$,
we have the following lemma.

\begin{lemma}\label{2.9}
If char$(\mathbb F)\neq 2$, then $\mathcal{C}_+$ and $\mathcal{C}_-$ are not tensor equivalent.
\end{lemma}

\begin{proof}
We first notice that if $\s$ is a permutation of $\{1,2\}$ such that the additive group endmorphism
$R\ra R$, $r_i\mapsto r_{\s(i)}$, is a ring automorphism, then $\s$ must be the identity
permutation of $\{1,2\}$.
Now assume that $\mathcal{C}_+$ and $\mathcal{C}_-$ are tensor equivalent. Then
by \cite[Theorem 3.22]{ChenZhang}, there exists an $\a\in\mathbb{F}^{\times}$ and a family of
invertible elements $\{\varphi_{i,j}\in M_{\mathbf{c}_{ij}}(A)|1\<i,j\<2\}$ such that
$\varphi_{1, i}=\varphi_{i, 1}=\a E_{\mathbf{e}_i}=\a e_i$ and
$$\begin{array}{rl}
&\sum\limits_{t=1}^2\sum\limits_{1\<k\<c_{ijt}}
a_{i,j,l}(X^{\mathbf{c}_{ij}}_{t,k}\ot e_l)\varphi_{t, l}
(Y^{\mathbf{c}_{ij}}_{t,k}\varphi_{i, j}\ot e_l)\\
=&\sum\limits_{t=1}^2\sum\limits_{1\<k\<c_{jlt}}
(e_i\ot X^{\mathbf{c}_{jl}}_{t,k})\varphi_{i, t}
(e_i\ot Y^{\mathbf{c}_{jl}}_{t,k}\varphi_{j, l})
a'_{i, j, l}\\
\end{array}$$
for all $1\<i,j,l\<2$. Since $\mathbf{c}_{22}=\mathbf{e}_1$, $\varphi_{2,2}=\b e_1$
for some $\b\in\mathbb{F}^{\times}$. Putting $i=j=l=2$ in the above equation, one obtains
$$e_2(e_1\ot e_2)(\a e_2)(e_1(\b e_1)\ot e_2)=(e_2\ot e_1)(\a e_2)(e_2\ot e_1(\b e_1))(-e_2),$$
which implies $\a\b e_2=-\a\b e_2$. This is impossible since char$(\mathbb F)\neq 2$.
\end{proof}

Summarizing the above discussion, we have the following proposition.

\begin{proposition}\label{2.10}
Let $\widehat{\mathcal C}$ be a semisimple tensor category of rank two over $\mathbb F$,
and $V$ a simple object not isomorphic to the unit object $\mathbf 1$.
If $V\ot V\cong m\mathbf{1}$ with $m>0$, then $m=1$ and $\widehat{\mathcal C}$
is tensor equivalent to $\mathcal{C}_+$ or $\mathcal{C}_-$.
\end{proposition}

\begin{proof}
Follows from \cite[Corollary 3.15 and Theorem 3.25]{ChenZhang} and the above discussion.
\end{proof}

\subsection{\bf $m=0$ and $n>0$} In this case, $\mathbf{e}_2\ot\mathbf{e}_2=\mathbf{c}_{22}=(0,n)=n\mathbf{e}_2$
and $\mathbf{e}_2\ot\mathbf{e}_2\ot\mathbf{e}_2=(0,n^2)=n^2\mathbf{e}_2$.  Moreover,
$\left\{\begin{array}{ccc}
                                                                       1 & 2 & 2 \\
                                                                       2 & 1 & 1
                                                                     \end{array}
 \right\}=\left\{\begin{array}{ccc}
                                                                       2 & 1 & 2 \\
                                                                       2 & 1 & 2
                                                                     \end{array}\right\}=\left\{\begin{array}{ccc}
                                                                       2 & 2 & 1 \\
                                                                       1 & 1 & 2
                                                                     \end{array}\right\}=\left\{\begin{array}{ccc}
                                                                       2 & 2 & 2 \\
                                                                       2 & 1 & 2
                                                                     \end{array}\right\}=\left\{\begin{array}{ccc}
                                                                       2 & 2 & 1 \\
                                                                       2 & 2 & 1
                                                                     \end{array}\right\}=\left\{\begin{array}{ccc}
                                                                       2 & 2 & 2 \\
                                                                       2 & 2 & 1
                                                                     \end{array}\right\}=\left\{\begin{array}{ccc}
                                                                       2 & 2 & 1 \\
                                                                       2 & 2 & 2
                                                                     \end{array}\right\}=0$,
$a_{2,2,2}=\L_{33}e_2$ and $\L_{33}=\left\{\begin{array}{ccc}
                                                                       2 & 2 & 2 \\
                                                                       2 & 2 & 2
                                                                     \end{array}\right\}$.
Thus the equation (\ref{a222+1}) is equivalent to
\begin{equation}\label{a222+2}
(I_n\ot_{\mathbb F}\L_{33})(\L_{33}\ot_{\mathbb F}I_n)(I_n\ot_{\mathbb F}\L_{33})
=(\L_{33}\ot_{\mathbb F}I_n)P_{23} (\L_{33}\ot_{\mathbb F}I_n).
\end{equation}

where $P_{23}=I_n\ot_{\mathbb F}I_{[n]}$. In case $n=1$, one can check that $\L_{33}=I_1\in M_1(\mathbb F)$ is the unique matrix satisfying the equation (\ref{a222+2}).
In this case, $a_{2,2,2}=e_2$, $\mathcal C$ becomes a tensor category, denoted  $\mathcal{C}_1$,
with the associativity constraint $a$ induced by $\{a_{i,j,l}|1\<i,j,l\<2\}$ as before.
Thus, we have the following proposition.

\begin{proposition}\label{2.11}
Let $\widehat{\mathcal C}$ be a semisimple tensor category of rank two over $\mathbb F$,
and $V$ a simple object not isomorphic to the unit object $\mathbf 1$.
If $V\ot V\cong V$, then $\widehat{\mathcal C}$
is tensor equivalent to $\mathcal{C}_1$.
\end{proposition}

\begin{proof}
Follows from \cite[Corollary 3.15 and Theorem 3.25]{ChenZhang} and the above discussion.
\end{proof}

For $n>1$, put $\L_{33}=I_{[n]}$. Then it is straightforward to verify that the matrix  $\L_{33}$ satisfies the equation (\ref{a222+2}).
In this case, $a_{2,2,2}=I_{[n]}e_2$, and $\mathcal C$ becomes a tensor category, denoted by $\mathcal{C}_n$,
with the associativity constraint $a$ induced
by $\{a_{i,j,l}|1\<i,j,l\<2\}$ as before. Similarly, putting
$$\Lambda_{33}=\left(\begin{array}{ccccc}
I_{1,1}&I_{2,1}&\cdots&I_{n-1,1}&I_{n,1}\\
I_{2,2}&I_{3,2}&\cdots&I_{n,2}&I_{1,2}\\
\cdots&\cdots&\cdots&\cdots&\cdots\\
I_{n-1,n-1}&I_{n,n-1}&\cdots&I_{n-3,n-1}&I_{n-2,n-1}\\
I_{n,n}&I_{1,n}&\cdots&I_{n-2,n}&I_{n-1,n}\\
\end{array}\right).$$
Let $a'_{2,2,2}=\L_{33} e_2$ and $a'_{1,i,j}=a'_{i,1,j}=a'_{i,j,1}=E_{\mathbf{e}_i\ot\mathbf{e}_j}$
for $1\<i,j\<2$.
Then $\mathcal C$ is also a tensor category, denoted $\mathcal{C}'_n$, with the associativity constraint $a'$ induced by
$\{a'_{i,j,l}|1\<i,j,l\<2\}$ as above.  Similarly to Lemma \ref{2.9}, one can prove the following proposition.

\begin{proposition}\label{2.12}
If $n>1$, then the tensor categories $\mathcal{C}_n$ and $\mathcal{C}'_n$ are not tensor equivalent.
\end{proposition}

\begin{proof}
It can be shown by using \cite[Theorem 3.22]{ChenZhang}.	
\end{proof}

\subsection{$m>0$ and $n>0$}
In this case, $\mathbf{e}_2\ot\mathbf{e}_2=\mathbf{c}_{22}=(m,n)=m\mathbf{e}_1+n\mathbf{e}_2$
and $\mathbf{e}_2\ot\mathbf{e}_2\ot\mathbf{e}_2=(mn,m+n^2)=mn\mathbf{e}_1+(m+n^2)\mathbf{e}_2$.
Hence the equation (\ref{a222+1}) holds if and only if the following equations are satisfied:\\
\mbox{\hspace{0.4cm}}(i)  $I_m\ot_{\mathbb F}\L_{22}^2+(I_m\ot_{\mathbb F}\L_{23})(\L_1\ot_{\mathbb F}I_n)(I_m\ot_{\mathbb F}\L_{32})=I_{[m]}$,\\
\mbox{\hspace{0.3cm}}(ii) $I_m\ot_{\mathbb F}\L_{22}\L_{23}+(I_m\ot_{\mathbb F}\L_{23})(\L_1\ot_{\mathbb F}I_n)(I_m\ot_{\mathbb F}\L_{33})=0$,\\
\mbox{\hspace{0.2cm}}(iii) $I_m\ot_{\mathbb F}\L_{32}\L_{22}+(I_m\ot_{\mathbb F}\L_{33})(\L_1\ot_{\mathbb F}I_n)(I_m\ot_{\mathbb F}\L_{32})=0$,\\
\mbox{\hspace{0.2cm}}(iv) $I_m\ot_{\mathbb F}\L_{32}\L_{23}+(I_m\ot_{\mathbb F}\L_{33})(\L_1\ot_{\mathbb F}I_n)(I_m\ot_{\mathbb F}\L_{33})$\\
\mbox{\hspace{1cm}}$=(\L_1\ot_{\mathbb F}I_n)(I_m\ot_{\mathbb F}I_{[n]})(\L_1\ot_{\mathbb F}I_n)$,\\
\mbox{\hspace{0.3cm}}(v) $\L_1(\L_{22}\ot_{\mathbb F}I_n)\L_1
=(\L_{23}\ot_{\mathbb F}I_n)(I_n\ot_{\mathbb F}I_{[n]})(\L_{32}\ot_{\mathbb F}I_n)$,\\
\mbox{\hspace{0.2cm}}(vi) $\L_1(\L_{23}\ot_{\mathbb F}I_n)(I_n\ot_{\mathbb F}\L_{32})
=(\L_{22}\ot_{\mathbb F}I_n)I_{[m,n]}$,\\
\mbox{\hspace{0.1cm}}(vii) $\L_1(\L_{23}\ot_{\mathbb F}I_n)(I_n\ot_{\mathbb F}\L_{33})
=(\L_{23}\ot_{\mathbb F}I_n)(I_n\ot_{\mathbb F}I_{[n]})(\L_{33}\ot_{\mathbb F}I_n)$,\\
(viii) $(I_n\ot_{\mathbb F}\L_{23})(\L_{32}\ot_{\mathbb F}I_n)\L_1
=I^T_{[m,n]}(\L_{22}\ot_{\mathbb F}I_n)$,\\
\mbox{\hspace{0.2cm}}(ix)  $I_n\ot_{\mathbb F}\L_{22}^2+(I_n\ot_{\mathbb F}\L_{23})(\L_{33}\ot_{\mathbb F}I_n)(I_n\ot_{\mathbb F}\L_{32})=0$,\\
\mbox{\hspace{0.3cm}}(x) $I_n\ot_{\mathbb F}\L_{22}\L_{23}+(I_n\ot_{\mathbb F}\L_{23})(\L_{33}\ot_{\mathbb F}I_n)(I_n\ot_{\mathbb F}\L_{33})
=I^T_{[m,n]}(\L_{23}\ot_{\mathbb F}I_n)$,\\
\mbox{\hspace{0.2cm}}(xi) $(I_n\ot_{\mathbb F}\L_{33})(\L_{32}\ot_{\mathbb F}I_n)\L_1
=(\L_{33}\ot_{\mathbb F}I_n)(I_n\ot_{\mathbb F}I_{[n]})(\L_{32}\ot_{\mathbb F}I_n)$,\\
\mbox{\hspace{0.1cm}}(xii) $I_n\ot_{\mathbb F}\L_{32}\L_{22}+(I_n\ot_{\mathbb F}\L_{33})(\L_{33}\ot_{\mathbb F}I_n)(I_n\ot_{\mathbb F}\L_{32})
=(\L_{32}\ot_{\mathbb F}I_n)I_{[m,n]}$,\\
(xiii)  $I_n\ot_{\mathbb F}\L_{32}\L_{23}+(I_n\ot_{\mathbb F}\L_{33})(\L_{33}\ot_{\mathbb F}I_n)(I_n\ot_{\mathbb F}\L_{33})$\\
\mbox{\hspace{1cm}}$=(\L_{33}\ot_{\mathbb F}I_n)(I_n\ot_{\mathbb F}I_{[n]})(\L_{33}\ot_{\mathbb F}I_n)$.\\

Assume that $m=n=1$. Then $\L_1=(\a)$ and
$\L_2=\left(\begin{array}{cc}
\a_{22}&\a_{23}\\
\a_{32}&\a_{33}\\
\end{array}\right)$
for some $\a, \a_{ij}\in\mathbb F$ with $\a\neq 0$ and $\a_{22}\a_{33}-\a_{23}\a_{32}\neq0$.
In this case, a straightforward computation shows that the equations (i)-(xiii) are equivalent to
that $\a=1$,  $\a_{33}=-\a_{22}$, $\a_{23}\a_{32}=\a_{22}$ and $\a_{22}^2+\a_{22}=1$.
Let $f(x)=x^2+x-1\in\mathbb{F}[x]$.

In case char$(\mathbb F)=5$, the polynomial $f(x)$ has a unique root $2$ in $\mathbb F$.
Let $$a_{2,2,2}=\left(\begin{array}{ccc}
e_1&0&0\\
0&2e_2&e_2\\
0&2e_2&3e_2\\
\end{array}\right).$$
Then by the above discussion, the equation (\ref{a222}) is satisfied, and hence $\mathcal C$
becomes a tensor category, denoted by $\mathcal{C}_{(5)}$,
with the associativity constraint $a$ induced by $\{a_{i,j,l}|1\<i,j,l\<2\}$
as before. For any $\a\in\mathbb{F}^{\times}$, let
$$a'_{2,2,2}=\left(\begin{array}{ccc}
e_1&0&0\\
0&2e_2&\a e_2\\
0&2\a^{-1}e_2&3e_2\\
\end{array}\right)$$
and $a'_{1,i,j}=a'_{I,1,j}=a'_{i,j,1}=E_{\mathbf{e_i}\ot\mathbf{e_j}}$ for $1\<i,j\<2$.
Similarly, $\mathcal C$ also becomes a tensor category, denoted  $\mathcal{C}_{[\a]}$,
with the associativity constraint $a'$ induced by $\{a'_{i,j,l}|1\<i,j,l\<2\}$ as before.

\begin{proposition}\label{2.13}
Assume that char$(\mathbb F)=5$. Let $\widehat{\mathcal C}$ be a semisimple tensor category of rank two over $\mathbb F$,
and $V$ a simple object not isomorphic to the unit object $\mathbf 1$.
If $V\ot V\cong\mathbf{1}\oplus V$, then $\widehat{\mathcal C}$
is tensor equivalent to $\mathcal{C}_{(5)}$.
\end{proposition}

\begin{proof}
By \cite[Corollary 3.15 and Theorem 3.25]{ChenZhang} and the discussion above,
there exists a nonzero scale $\a$ in $\mathbb{F}$ such that $\widehat{\mathcal C}$
is tensor equivalent to $\mathcal{C}_{[\a]}$. Let
$a_{2,2,2}$ and $a'_{2,2,2}$ be the matrices given above,
and let $a'_{1,i,j}=a'_{i,1,j}=a'_{i,j,1}=E_{\mathbf{e_i}\ot\mathbf{e_j}}$ for $1\<i,j\<2$.
Then the associativity constraints of $\mathcal{C}_{(5)}$ and $\mathcal{C}_{[\a]}$
are induced by $\{a_{i,j,l}|1\<i,j,l\<2\}$ and $\{a'_{i,j,l}|1\<i,j,l\<2\}$, respectively.
Let $\eta(1,1)=e_1$, $\eta(1,2)=\eta(2,1)=e_2$ and
$\eta(2,2)=\left(\begin{array}{cc}
\a e_1&0\\
0&\a e_2\\
\end{array}\right)$.
Then it is straightforward to check that the two families $\{a_{i,j,l}|1\<i,j,l\<2\}$ and $\{a'_{i,j,l}|1\<i,j,l\<2\}$
are equivalent via the family of matrices $\{\eta(i,j)\in M_{\mathbf{c}_{ij}}(A)|1\<i,j\<2\}$
(see \cite[Definition 3.19]{ChenZhang}).
Thus, it follows from \cite[Proposition 3.20]{ChenZhang} that ${\mathcal C}_{(5)}$
is tensor equivalent to $\mathcal{C}_{[\a]}$, hence
$\widehat{\mathcal C}$ is tensor equivalent to $\mathcal{C}_{(5)}$.
\end{proof}

In case char$(\mathbb F)\neq 5$, the polynomial $f(x)$ has two distinct roots $\o_1$ and $\o_2$
in $\mathbb F$. If char$(\mathbb F)\neq 2$, then $\o_1=\frac{-1+\sqrt{5}}{2}$ and $\o_2=\frac{-1-\sqrt{5}}{2}$. If char$(\mathbb F)=2$, then  $\o_1$ is a $3^{th}$ primitive root of unity and $\o_2=\o_1^2$.
Let $$a_{2,2,2}=\left(\begin{array}{ccc}
e_1&0&0\\
0&\o_1e_2&e_2\\
0&\o_1e_2&-\o_1e_2\\
\end{array}\right).$$
Then by the discussion before, $\mathcal C$ becomes a tensor category, denoted  $\mathcal{C}_{(0)}$,
with the associativity constraint $a$ induced by $\{a_{i,j,l}|1\<i,j,l\<2\}$.
Similarly, let $$a'_{2,2,2}=\left(\begin{array}{ccc}
e_1&0&0\\
0&\o_2e_2&e_2\\
0&\o_2e_2&-\o_2e_2\\
\end{array}\right)$$
and $a'_{1,i,j}=a'_{i,1,j}=a'_{i,j,1}=E_{\mathbf{e_i}\ot\mathbf{e_j}}$ for $1\<i,j\<2$.
Then $\mathcal C$ also becomes a tensor category, denoted $\mathcal{C}'_{(0)}$
with the associativity constraint $a'$ induced by $\{a'_{i,j,l}|1\<i,j,l\<2\}$.
For any $0\neq\a\in\mathbb F$, let
$$a''_{2,2,2}=\left(\begin{array}{ccc}
e_1&0&0\\
0&\o_1e_2&\a e_2\\
0&\a^{-1}\o_1e_2&-\o_1e_2\\
\end{array}\right)$$
and $a''_{1,i,j}=a''_{i,1,j}=a''_{i,j,1}=E_{\mathbf{e_i}\ot\mathbf{e_j}}$ for $1\<i,j\<2$.
Then by an argument similar to the proof of Proposition \ref{2.13}, one can check that
$\{a_{i,j,l}|1\<i,j,l\<2\}$ and $\{a''_{i,j,l}|1\<i,j,l\<2\}$ are equivalent.
Simialrly, if we set
$$a'''_{2,2,2}=\left(\begin{array}{ccc}
e_1&0&0\\
0&\o_2e_2&\a e_2\\
0&\a^{-1}\o_2e_2&-\o_2e_2\\
\end{array}\right)$$
and $a'''_{1,i,j}=a'''_{i,1,j}=a'''_{i,j,1}=E_{\mathbf{e_i}\ot\mathbf{e_j}}$ for $1\<i,j\<2$,
then $\{a'_{i,j,l}|1\<i,j,l\<2\}$ and $\{a'''_{i,j,l}|1\<i,j,l\<2\}$ are equivalent.
Thus, we have the following proposition.

\begin{proposition}\label{2.14}
Assume that char$(\mathbb F)\neq 5$. Let $\widehat{\mathcal C}$ be a semisimple tensor category of rank two over $\mathbb F$,
and $V$ a simple object not isomorphic to the unit object $\mathbf 1$.
If $V\ot V\cong\mathbf{1}\oplus V$, then $\widehat{\mathcal C}$
is tensor equivalent to $\mathcal{C}_{(0)}$ or $\mathcal{C}'_{(0)}$.
Moreover, $\mathcal{C}_{(0)}$ and $\mathcal{C}'_{(0)}$ are not tensor equivalent.
\end{proposition}

\begin{proof}
The first statement follows from an argument similar to the proof of Proposition \ref{2.13} and the above discussion.
By \cite[Theorem 3.22]{ChenZhang}, one can see that $\mathcal{C}_{(0)}$ and $\mathcal{C}'_{(0)}$
are tensor equivalent if and only if $\{a_{i,j,l}|1\<i,j,l\<2\}$ and $\{a'_{i,j,l}|1\<i,j,l\<2\}$ are equivalent.
However, it is easy to check that $\{a_{i,j,l}|1\<i,j,l\<2\}$ and $\{a'_{i,j,l}|1\<i,j,l\<2\}$ are not equivalent.
Therefore,  $\mathcal{C}_{(0)}$ and $\mathcal{C}'_{(0)}$ are not tensor equivalent.
\end{proof}

Now assume that $m>1$. Let $G=(\L_1\ot_{\mathbb F}I_n)(I_m\ot_{\mathbb F}\L_{32})\in M_{mn^2\times m^2}(\mathbb F)$.
Then (i) and (iii) are equivalent to
$$(I_m\ot_{\mathbb F}\L_2)
\left(\begin{array}{c}
I_m\ot_{\mathbb F}\L_{22}\\
G\\
\end{array}\right)
=\left(\begin{array}{c}
I_{[m]}\\
0\\
\end{array}\right).$$
Let $\L_2=(\a_{ij})$ with $\a_{ij}\in\mathbb F$ and
$G=\left(\begin{array}{cccc}
G_{1,1}&G_{1,2}&\cdots&G_{1,m}\\
G_{2,1}&G_{2,2}&\cdots&G_{2,m}\\
\cdots&\cdots&\cdots&\cdots\\
G_{n^2,1}&G_{n^2,2}&\cdots&G_{n^2,m}\\
\end{array}\right)$ with $G_{i,j}\in M_m(\mathbb F)$.
Then $\L_{22}=\left(\begin{array}{cccc}
\a_{11}&\a_{12}&\cdots&\a_{1m}\\
\a_{21}&\a_{22}&\cdots&\a_{2m}\\
\cdots&\cdots&\cdots&\cdots\\
\a_{m1}&\a_{m2}&\cdots&\a_{mm}\\
\end{array}\right)$.
Thus, for any $1\<s\<m+n^2$ and $1\<t\<m$, it follows from the matrix equation above that
$$\sum\limits_{j=1}^m\a_{sj}\a_{jt}I_m+\sum\limits_{j=m+1}^{m+n^2}\a_{sj}G_{j-m,t}
=\left\{\begin{array}{ll}
I_{t,s},& \mbox{ \rm if }1\<s\<m,\\
0,&\mbox{ \rm if }m+1\<s\<m+n^2.\\
\end{array}\right.$$
Hence for any $1\<p,q\<m$, by comparing the $(p,q)$-entries of the matrices on the both sides of the above equation, one gets
\begin{equation}\label{Eq1}
\sum\limits_{j=1}^m\a_{sj}\a_{jt}\d_{pq}+\sum\limits_{j=m+1}^{m+n^2}\a_{sj}G_{j-m,t}(p,q)
=\left\{\begin{array}{ll}
\d_{tp}\d_{sq},& \mbox{ \rm if }1\<s\<m,\\
0,&\mbox{ \rm if }m+1\<s\<m+n^2,\\
\end{array}\right.
\end{equation}
where $\d_{pq}$'s are Kronecker symbols, $G_{j-m,t}(p,q)$ is the $(p,q)$-entry of the matrix $G_{j-m,t}$.
Let $1\<k\<m$.
Putting $t=p=q=k$ in the equation (\ref{Eq1}), one finds that
$$(\a_{1k}, \a_{2k},\cdots,\a_{mk},G_{1,k}(k,k),G_{2,k}(k,k),\cdots,G_{n^2,k}(k,k))^T$$
is a solution of the system of linear equations $\L_2X=\e_{m+n^2,k}^T$.
Now assume that $m>1$. Then one can choose an integer $l$ with $1\<l\<m$
such that $l\neq k$.  Putting $t=p=l$ and $q=k$ in the equation (\ref{Eq1}), one finds that
$$(0,0,\cdots,0,G_{1,l}(l,k),G_{2,l}(l,k),\cdots,G_{n^2,l}(l,k))^T$$
is also a solution of the system of linear equations $\L_2X=\e_{m+n^2,k}^T$.
However, the solution of the system of linear equations $\L_2X=\e_{m+n^2,k}^T$
is unique since its coefficient matrix $\L_2$ is invertible. This implies that
$\a_{1k}=\a_{2k}=\cdots=\a_{mk}=0$ and $\L_{22}=0$.
Again by the invertibility of $\L_2$, the ranks of the matrices $\L_{23}$
and  $\L_{32}$ are both equal to $m$.
Then by the equation (vi), one gets $(\L_{23}\ot_{\mathbb F}I_n)(I_n\ot_{\mathbb F}\L_{32})=0$,
which implies $2mn\<n^3$, and hence $2m\<n^2$.

\begin{proposition}\label{2.15}
Let $\widehat{\mathcal C}$ be a semisimple tensor category of rank two over $\mathbb F$,
and $V$ a simple object not isomorphic to the unit object $\mathbf 1$.
Suppose that $V\ot V\cong m\mathbf{1}\oplus nV$ with $m>0$ and $n>0$.
If $m>1$, then $2m\<n^2$. Furthermore, if $m=2$, then $n>2$.
\end{proposition}

\begin{proof}
The first statement follows from the above discussion. Now let $m=2$. Then $n\>2$.
Suppose $n=2$. Then $\L_{23}\in M_{2\times 4}(\mathbb F)$ and $\L_{32}\in M_{4\times 2}(\mathbb F)$.
By the discussion above, the ranks of $\L_{23}$ and $\L_{32}$ are both equal to 2.
Let $\L_{32}=\left(\begin{array}{cc}
\b_{11}&\b_{12}\\
\b_{21}&\b_{22}\\
\b_{31}&\b_{32}\\
\b_{41}&\b_{42}\\
\end{array}\right)$.
Then by the equations (v) and (vi), the following vectors are solutions of the system of
linear equations $\L_{23}X=0$:
 $$\left(\begin{array}{c}
\b_{1i}\\
\b_{2i}\\
0\\
0\\
\end{array}\right),
\left(\begin{array}{c}
0\\
0\\
\b_{1i}\\
\b_{2i}\\
\end{array}\right),
\left(\begin{array}{c}
\b_{1i}\\
0\\
\b_{2i}\\
0\\
\end{array}\right),
\left(\begin{array}{c}
0\\
\b_{1i}\\
0\\
\b_{2i}\\
\end{array}\right),
\left(\begin{array}{c}
\b_{3i}\\
\b_{4i}\\
0\\
0\\
\end{array}\right),
\left(\begin{array}{c}
0\\
0\\
\b_{3i}\\
\b_{4i}\\
\end{array}\right),
\left(\begin{array}{c}
\b_{3i}\\
0\\
\b_{4i}\\
0\\
\end{array}\right),
\left(\begin{array}{c}
0\\
\b_{3i}\\
0\\
\b_{4i}\\
\end{array}\right),$$
where $1\<i\<2$. Since $(\b_{1i},\b_{2i},\b_{3i},\b_{4i})\neq(0,0,0,0)$, the rank of the above vector group
is larger than or equal to 3, which contradicts to the fact that the rank of the matrix $\L_{23}$ is 2.
Therefore, $n>2$.
\end{proof}

Ostrik proved that there are only 4 fusion categories of rank 2 \cite{Ostrik03-2}. By above discussion, we have following remark.

\begin{remark}
The fusion categories of rank two are $\mathcal{C}_+$, $\mathcal{C}_-$, $\mathcal{C}_{(0)}$ and $\mathcal{C}'_{(0)}$.
\end{remark}
 The following table display our main results.
\begin{table}[ht!]

\center
\begin{tabular}{|c|c|c|} \hline
 \multirow{2}*{$m=1,n=0$} & \multirow{2}*{$\mathcal{C}_+ $, $\mathcal{C}_-$}& $\mathcal{C}_+ \ncong \mathcal{C}_-$ , if char$\Bbbk\neq 2$\\ \cline{3-3}
 \multirow{2}*{} &  \multirow{2}*{}& $\mathcal{C}_+ \cong \mathcal{C}_-$, if char$\Bbbk=2$\\ \hline
  $m>1,n=0$&  \multicolumn{2}{c|}{No}  \\ \hline
  $m=0,n=1$ &  \multicolumn{2}{c|}{$\mathcal{C}_1$}  \\ \hline
    $m=0,n>1$ &  \multicolumn{2}{c|}{$\mathcal{C}_n\ncong\mathcal{C}'_n$}  \\ \hline
  $m=1,n=1$ & {$\mathcal{C}_{(5)}$, if char$\Bbbk$=5}& {$\mathcal{C}_{(0)}\ncong \mathcal{C}'_{(0)}$, if char$\Bbbk \neq 5$}\\ \hline
        $m>1$, $2m>n^2$ &  \multicolumn{2}{c|}{No}  \\ \hline
          $m=2$,$n=2$ &  \multicolumn{2}{c|}{No}  \\ \hline
  \end{tabular}
   \end{table}

\centerline{ACKNOWLEDGMENTS}

This work is supported by NNSF of China (Nos. 12071412,12201545).

\end{document}